\documentclass{llncs}
\usepackage{amsmath}
\usepackage{amssymb}
\usepackage{mathrsfs}
\usepackage{mathabx}
\usepackage{graphicx}
\usepackage{xcolor}
\usepackage[normalsize]{subfigure}
\usepackage{hyperref}

\newcommand{\bs}[1]{\boldsymbol{#1}}

\newcommand{\FFTH}{\mathrm{FFTH}}
\newcommand{\CG}{\mathrm{CG}}

\newcommand{\scal}[2]{\bigl(#1,#2\bigr)}

\newcommand{\ZN}{\widebar{\set{Z}}^{\VN}}
\newcommand{\meas}[1]{\left\lvert#1\right\rvert}
\newcommand{\Matlab}{\textsc{Matlab}\textsuperscript{\textregistered} }

\newcommand{\set}[1]{\mathbb{#1}}

\newcommand{\Vxi}{\ensuremath{\V{\xi}}}

\newcommand{\del}{\ensuremath{\delta}}
\newcommand{\alp}{\ensuremath{\alpha}}

\newcommand{\vek}[1]{\bs{#1}} 
\newcommand{\ee}{\vek{e}}
\newcommand{\je}{\vek{j}}
\newcommand{\se}{L}
\newcommand{\tse}{\vek{\se}}

\newcommand{\x}{\vek{x}}
\newcommand{\y}{\vek{y}}

\newcommand{\LL}{L} 
\newcommand{\tL}{\vek{\LL}}

\newcommand{\FT}[1]{\widehat{#1}}

\newcommand{\DFT}{\mtrx{F}}

\newcommand{\ite}[1]{^{(#1)}}

\newcommand{\B}[1]{\boldsymbol{#1}}
\newcommand{\M}[1]{\mat{#1}}
\newcommand{\MB}[1]{\boldsymbol{\mat{#1}}}
\newcommand{\m}[1]{\bs{#1}}
\newcommand{\mat}[1]{\mathsf{#1}} 
\newcommand{\mtrx}[1]{\mathsf{#1}} 
\newcommand{\V}[1]{\bs{#1}}

\newcommand{\ML}{\ensuremath{\M{L}}}
\newcommand{\mL}{\ensuremath{\m{L}}}

\newcommand{\mI}{\ensuremath{\m{I}}}

\newcommand{\Vx}{\ensuremath{\V{x}}}

\newcommand{\Vy}{\ensuremath{\V{y}}}
\newcommand{\Vr}{\ensuremath{\V{r}}}

\newcommand{\Vk}{\ensuremath{\V{k}}}
\newcommand{\Vm}{\ensuremath{\V{m}}}

\newcommand{\MA}{\ensuremath{\M{A}}}
\newcommand{\MBA}{\ensuremath{\B{\M{A}}}}
\newcommand{\MBF}{\ensuremath{\B{\M{F}}}}
\newcommand{\MBFi}{\ensuremath{\B{\M{F}}^{-1}}}
\newcommand{\MBe}{\ensuremath{\B{\M{e}}}}
\newcommand{\Me}{\ensuremath{\M{e}}}
\newcommand{\MBPE}{\ensuremath{\B{\M{P}}_{\cE}}}
\newcommand{\MBL}{\ensuremath{\B{\M{L}}}}

\newcommand{\Vo}{\ensuremath{{\V{0}}}}

\newcommand{\MBx}{\MB{x}}

\newcommand{\mG}{\ensuremath{{\m{\Gamma}}}}

\newcommand{\MhG}{\ensuremath{{\M{\hat{\Gamma}}}}}
\newcommand{\MBhG}{\ensuremath{{\B{\M{\hat{\Gamma}}}}}}
\newcommand{\mhG}{\ensuremath{{\m{\hat{\Gamma}}}}}
\newcommand{\hG}{\ensuremath{{\hat{\Gamma}}}}
\newcommand{\kk}{\ensuremath{{\V{k}}}}

\newcommand{\MBB}{\ensuremath{{\B{\M{B}}}}}
\newcommand{\MBI}{\ensuremath{{\B{\M{I}}}}}

\newcommand{\VN}{\ensuremath{{\V{N}}}}
\newcommand{\Ve}{\ensuremath{\V{e}}}

\newcommand{\Vj}{\ensuremath{\V{j}}}

\newcommand{\Eqref}[1]{Eq.~\eqref{#1}}

\newcommand{\Fref}[1]{Fig.~\ref{#1}}
\newcommand{\JV}[1]{\textcolor{red}{#1}}
\newcommand{\PUC}{\mathcal{Y}}
\newcommand{\lPUC}{Y}

\newcommand{\beq}{\begin{equation*}}
\newcommand{\eeq}{\end{equation*}}
\newcommand{\beql}[1]{\begin{equation}\label{#1}}
\newcommand{\eeql}{\end{equation}}
\newcommand{\beqa}{\begin{align}}
\newcommand{\eeqa}{\end{align}}
\newcommand{\bml}{\begin{multline*}}
\newcommand{\eml}{\end{multline*}}
\newcommand{\bmll}[1]{\begin{multline}\label{#1}} 
\newcommand{\emll}{\end{multline}}
\newcommand{\blem}{\begin{lem}}
\newcommand{\elem}{\end{lem}}
\newcommand{\bleml}[1]{\begin{lem}\label{#1}}
\newcommand{\bpro}{\begin{proof}}
\newcommand{\epro}{\end{proof}}
\newcommand{\btheo}[1]{\begin{theorem}\label{#1}}
\newcommand{\etheo}{\end{theorem}}
\newcommand{\bcor}{\begin{corollary}}
\newcommand{\bcorl}[1]{\begin{corollary}\label{#1}}
\newcommand{\ecor}{\end{corollary}}
\newcommand{\bdf}{\begin{definition}}
\newcommand{\edf}{\end{definition}}
\newcommand{\bM}[1]{\begin{pmatrix}}
\newcommand{\eM}[1]{\end{pmatrix}}
\newcommand{\bspl}{\begin{split}}
\newcommand{\espl}{\end{split}}
\newcommand{\bal}{\begin{align}}
\newcommand{\eal}{\end{align}}

\newcommand{\cE}{\ensuremath{\mathcal{E}}}
\newcommand{\cEp}{\ensuremath{\mathcal{E^{\perp}}}}

\DeclareMathOperator{\D}{d\!}

\newcommand{\imu}{\mathrm{i}}		

\newcommand{\Kryl}[3]{\mathscr{K}_{#1}(#2,#3)}

\makeatletter
\renewcommand{\JV}[1]{#1}

\begin{document}

\mainmatter             
\title{Analysis of a Fast Fourier Transform Based Method for Modeling of
Heterogeneous Materials}
\titlerunning{Analysis of a FFT Based Method for Modeling of
Heterogeneous Materials} 

\author{Jaroslav Vond\v{r}ejc\inst{1} \and Jan Zeman\inst{1} \and
Ivo Marek\inst{2}}
\authorrunning{Jaroslav Vond\v{r}ejc et al.} 

\institute{Czech Technical University in Prague,\\ Faculty of Civil Engineering, Department of Mechanics\\
\email{vondrejc@gmail.com},
\and
Czech Technical University in Prague,\\
Faculty of Civil Engineering, Department of Mathematics}

\maketitle

\begin{abstract}
The focus of this paper is on the analysis of the Conjugate Gradient method
applied to a non-symmetric system of linear equations, arising from a Fast
Fourier Transform-based homogenization method due to Moulinec and
Suquet~\cite{Moulinec:1994:FNMC}. Convergence of the method is proven by
exploiting a certain projection operator reflecting physics of the underlying
problem. These results are supported by a numerical example, demonstrating
significant improvement of the Conjugate Gradient-based scheme over the original
Moulinec-Suquet algorithm.
\end{abstract}

\section{Introduction}\label{sec:introduction}

The last decade has witnessed a rapid development in advanced experimental
techniques and modeling tools for microstructural characterization,
typically provided in the form of pixel- or voxel-based geometry. Such data now
allow for the design of bottom-up predictive models of the overall behavior for
a wide range of engineering materials. Of course, such step necessitates the
development of specialized algorithms, capable of handling large-scale
voxel-based data in an efficient manner. In the engineering community, perhaps
the most successful solver meeting these criteria was proposed by Moulinec and
Suquet in~\cite{Moulinec:1994:FNMC}. The algorithm is based on the Neumann
series expansion of the \JV{inverse of} an operator arising in the associated
Lippmann-Schwinger equation and exploits the Fast Fourier Transform to evaluate
the action of the operator efficiently \JV{for voxel-based data}. In our
recent work~\cite{Zeman:2010:AFFT}, we have offered a \JV{new} approach to the
Moulinec-Suquet scheme, by exploiting the trigonometric collocation method due
to Saranen and Vainikko~\cite{Saranen:2002:PIP}. Here, the Lippman-Schwinger
equation is projected to a space of trigonometric polynomials to yield a
non-symmetric system of linear equations, see Section~\ref{sec:problem_setting}
below.  Quite surprisingly, numerical experiments revealed that the system can be
efficiently solved \JV{using the} standard Conjugate Gradient algorithm. The analysis of
this phenomenon, as presented in Section~\ref{sec:convergence}, is at the heart
of this contribution. The obtained results are further supported by a numerical
example in Section~\ref{sec:example} and summarized in
Section~\ref{sec:conclusions}.

The following notation is used throughout the paper. Symbols $a$, $\vek{a}$ and
$\vek{A}$ denote scalar, vector and second-order tensor quantities,
respectively, with Greek subscripts used when referring to the corresponding
components, e.g. $A_{\alpha\beta}$. The outer product of two vectors is denoted
as $\vek{a} \otimes \vek{a}$, whereas $\vek{a} \cdot \vek{b}$ or $\vek{A} \cdot
\vek{b}$ represents the single contraction between vectors~(or tensors). A
multi-index notation is employed, in which $\set{R}^{\VN}$ with $\VN = (N_1,
\ldots, N_d)$ represents $\set{R}^{N_1 \times \cdots \times N_d}$ and $|\VN|$
abbreviates $\prod_{\alpha=1}^d N_\alpha$. Block matrices are denoted by capital
letters typeset in a bold serif font, e.g. $\MBA\in\set{R}^{d\times d\times
\VN\times\VN}$, and the superscript and subscript indexes are used to refer to
the components, such that
$
\MB{A} 
= 
[\M{A}^{\Vk\Vm}_{\alp\beta}]_{\alp,\beta=1,\dotsc,d}^{\Vk,\Vm\in\ZN}
$
with 
$$
\ZN  
=   
\left\{ 
  \Vk \in \set{Z}^d 
  :
  -\frac{N_\alpha}{2} < 
  k_\alpha 
  \leq \frac{N_\alpha}{2}, 
  \alpha = 1, \ldots, d
\right\}.
$$
%
Sub-matrices of $\MB{A}$ are denoted as
\begin{eqnarray*}
\MBA_{\alp\beta} 
= 
\left[
  \MA_{\alp\beta}^{\Vk\Vm}
\right]^{\Vk,\Vm\in\ZN}
\in
\set{R}^{\VN\times\VN},
&&
\MBA^{\Vk\Vm} 
= 
\left[
  \MA_{\alp\beta}^{\Vk\Vm}
\right]_{\alp,\beta=1,\dotsc,d}
\in\set{R}^{d\times d}
\end{eqnarray*}
for $\alp,\beta=1,\dotsc,d$ and $\Vk,\Vm\in\ZN$.
Analogously, the block vectors are denoted by lower case letters,
e.g. $\MB{e}\in\set{R}^{d\times \VN}$ and the matrix-by-vector multiplication
is defined as
\begin{equation}\label{eq:MB_multiplication}
  \left[\MA\Me\right]_{\alp}^{\Vk}
  =
  \sum_{\beta=1}^d \sum_{\Vm\in\ZN}
  \MA_{\alp\beta}^{\Vk\Vm}
  \Me_{\beta}^{\Vm}
  \in \set{R}^{d \times \VN},
\end{equation}
with $\alp=1,\dotsc,d$ and $\Vk\in\ZN$.

\section{Problem setting}\label{sec:problem_setting}
Consider a composite material represented by a periodic unit cell
$$\PUC = \prod_{\alpha=1}^d ( -\lPUC_\alpha, \lPUC_\alpha ) \subset
\set{R}^d.$$ In the context of linear electrostatics, the associated
unit cell problem reads as
\begin{align}\label{eq:problem_def_1}
  \V{\nabla} \times \Ve( \Vx ) &= \Vo, & 
  \V{\nabla} \cdot \Vj( \Vx ) &= \Vo,  & 
  \je( \x ) &= \tse( \x ) \cdot \ee( \x ), &
  \x &\in \PUC
\end{align}
where $\ee$ is a $\PUC$-periodic vectorial electric field, $\Vj$
denotes the corresponding vector of electric current and $\mL$ is a
second-order positive-definite tensor of electric conductivity. In
addition, the field $\Ve$ is subject to a constraint
\begin{equation}\label{eq:problem_def_2}
   \langle \Ve(\Vx) \rangle 
   := 
   \frac{1}{\meas{\PUC}} 
   \int_{\PUC} \Ve( \Vx ) 
     \D{\Vx} 
   = 
   \Ve^0,
\end{equation}
where $\meas{\PUC}$ denotes the $d$-dimensional measure of $\PUC$ and $\Ve^0
\neq \vek{0}$ a prescribed macroscopic electric field.

The original problem~\eqref{eq:problem_def_1}--\eqref{eq:problem_def_2} is
then equivalent to the periodic Lippmann-Schwinger integral equation, formally
written as
\begin{equation}\label{eq:LS_real}
  \ee( \x ) 
  +
  \int_\PUC 
  \mG( \Vx - \Vy;\mL^0 ) 
  \cdot 
  \Bigl( 
    \mL( \y )-\mL^0\Bigr) \cdot \ee( \Vy )  
  \D{\Vy}
  = 
  \ee^0,\quad
  \x \in \PUC,
\end{equation}
where \JV{$\mL^0\in\set{R}^{d\times d}$} denotes a homogeneous reference medium.
The operator $\mG(\Vx,\mL^0)$ is derived from the Green's function of the
problem~\eqref{eq:problem_def_1}--\eqref{eq:problem_def_2} with $\tL( \x ) =
\tL^0$ and can be simply expressed in the Fourier space
\begin{equation}\label{eq:Gamma}
  \mhG(\Vk;\mL^0)=
\begin{cases}
  \Vo  &\Vk=\Vo
\\
\displaystyle
\frac{%
\Vxi \otimes \Vxi
}{%
\Vxi \cdot \mL^0 \cdot \Vxi}
  &
  \Vxi(\Vk)
  =
  \left(
    \frac{k_{\alp}}{Y_{\alp}}
  \right)_{\alp=1}^d; 
  \Vk\in\set{Z}^d \setminus \Vo.
\end{cases}
\end{equation}
Operator $\FT{f} = \FT{f}( \Vk )$ stands for the Fourier coefficient of $f( \x
)$ for the $\Vk$-th frequency given by
\begin{eqnarray}\label{eq:fourier_coeff_def}
  \FT{f}( \Vk ) 
  = 
  \int_\PUC f( \x ) \varphi_{-\Vk} ( \x ) 
  \D{x}, 
&&
  \varphi_{ \Vk }( \vek{x} ) 
  = 
  | \PUC |^{-\frac{1}{2}}
  \exp\left( 
    \imu \pi \sum_{\alpha=1}^d \frac{x_\alpha k_\alpha}{\lPUC_\alpha} 
  \right),
\end{eqnarray}
"$\imu$'' is the imaginary unit ($\imu^2 = -1$).
We refer to~\cite{Zeman:2010:AFFT,Milton:2002:TC} for additional details.
\JV{Note that the linear electrostatics serves here as a model problem; the
framework can be directly extended to e.g. elasticity~\cite{Smilauer:2006:MBM},
(visco-)plasticity~\cite{Prakash:2009:SMB} or to
multiferroics~\cite{Brenner:2010:RMC}.}

\subsection{Discretization via trigonometric collocation}\label{sec:discretization}

\JV{The numerical solution} of the Lippmann-Schwinger equation is based on a
discretization of a unit cell $\PUC$ into a regular periodic grid with
$N_1 \times \cdots \times N_d$ nodal points and grid spacings $\vek{h}
= ( 2\lPUC_1/N_1, \ldots, 2 \lPUC_d/N_d )$. The searched field $\ee$
in \eqref{eq:LS_real} is approximated by a trigonometric polynomial
$\Ve^\VN$ in the form~(cf.~\cite[Chapter 10]{Saranen:2002:PIP})
\begin{eqnarray}\label{eq:trig_pol_def}
  \Ve( \Vx ) 
  \approx 
  \Ve^\VN( \Vx ) 
  = 
  \sum_{\vek{\Vk} \in \ZN}
  \hat{\MBe}^{\Vk} \varphi_{\Vk}( \Vx ), 
  && 
  \Vx \in \PUC,
\end{eqnarray}
where $\hat{\MBe}^{\Vk}=(\hat{\Me}^{\Vk}_{\alp})_{\alp=1,\dotsc,d}$ designates
the Fourier coefficients defined in~\eqref{eq:fourier_coeff_def}. \JV{Notice
that the trigonometrical polynomials are uniquely determined by a regular
grid data, which makes them well-suited to problems with pixel- or voxel-based
computations.}

The trigonometric collocation method is based on the projection of
the Lippmann-Schwinger equation~\eqref{eq:LS_real} \JV{onto} the space of the
trigonometric polynomials
\begin{equation}
\mathcal{T}^\VN
=
\Bigl\{
\sum_{\kk \in \ZN} c_\kk
  \varphi_\kk, c_\kk \in \set{C} 
\Bigr\},
\end{equation}
leading to a to linear system in the form, cf.~\cite{Zeman:2010:AFFT}
\begin{eqnarray}\label{eq:linear_system}
(\MBI+\MBB) \MB{e} 
= \MB{e^0},
&&
\MBB 
=
\MBFi\MBhG\MBF(\MBL-\MB{L^0}),
\end{eqnarray}
where
$\MBe=\left(\M{e}_{\alp}^{\Vk}\right)_{\alp=1,\dotsc,d}^{\Vk\in\ZN}\in\set{R}^{d
\times \VN}$ is the \JV{unknown} vector,
$\MBI=\left[\del_{\alp\beta}\del_{\Vk\Vm}\right]_{\alp,\beta=1,\dotsc,d}^{\Vk\Vm\in\ZN}\in\set{R}^{d\times
d\times\VN\times\VN}$ is the identity matrix, expressed as the product of the
Kronecker delta functions $\del_{\alp\beta}$ and $\del_{\Vk\Vm}$, and
$\MB{e^0}=(e^0_{\alp})_{\alp=1,\dotsc,d}^{\Vk\in\ZN}\in\set{R}^{d \times \VN}$.

All the matrices in~\eqref{eq:linear_system} exhibit a block-diagonal structure.
In particular,
\begin{eqnarray*}
\MBhG
=
\left[
  \del_{\Vk\Vm}\MhG_{\alp\beta}^{\Vk\Vm}
\right]_{\alp,\beta=1,\dotsc,d}^{\Vk,\Vm\in\ZN},
& \displaystyle
\MBL 
=
\left[
  \del_{\Vk\Vm}\ML_{\alp\beta}^{\Vk\Vm}
\right]_{\alp,\beta=1,\dotsc,d}^{\Vk,\Vm\in\ZN},
&
\MB{L^0} 
=
\left[
  \del_{\Vk\Vm}\mat{L^0}_{\alp\beta}
\right]_{\alp,\beta=1,\dotsc,d}^{\Vk,\Vm\in\ZN},
\end{eqnarray*}
with $\MhG_{\alp\beta}^{\Vk\Vk} = \hG_{\alp\beta}(\Vk;\mL^0)$,
$\ML_{\alp\beta}^{\Vk\Vk} = L_{\alp\beta}(\Vk)$ and $(\mat{L^0})_{\alp\beta} =
L^0_{\alp\beta}$. The matrix $\MBF$ implements the Discrete Fourier Transform
and is defined as
\begin{align}
  \MB{F} &=
  \left[\del_{\alp\beta}\mat{F}^{\Vk\Vm}\right]_{\alp,\beta=1,\dotsc,d}^{\Vk,\Vm\in\ZN}, &
\DFT^{\Vk\V{m}} 
  & =  \frac{| \PUC |^{\frac{1}{2}}}{\prod_{\alp=1}^dN_{\alp}}  \exp   \left( - \sum_{\alpha=1}^{d} 
    2 \pi \imu \frac{k_\alpha m_\alpha}{N_\alpha}  \right),
\end{align}
with $\MBFi$ representing the inverse transform.

It follows from Eq.~\eqref{eq:MB_multiplication} that the cost of multiplication
by $\MBB$ is dominated by the action of $\MB{F}$ and $\MBFi$, which can be
performed in $O(\meas{\VN} \log \meas{\VN} )$ operations by the Fast Fourier
Transform techniques. This makes the system~\eqref{eq:linear_system} well-suited
\JV{for applying some iterative solution technique.} In particular, the
original Fast Fourier Transform-based Homogenization scheme formulated by
Moulinec and Suquet in~\cite{Moulinec:1994:FNMC} is based on the Neumann
expansion of the matrix inverse $(\MB{I} + \MB{B})^{-1}$, so as to yield the
$m$-th iterate in the form
\begin{equation}\label{eq:orig_fft}
\MB{e}\ite{m} 
=
\sum_{j=0}^{m} 
\left( -\MB{B} \right)^j
\MB{e^0}.
\end{equation}
As indicated earlier, our numerical experiments~\cite{Zeman:2010:AFFT} suggest
that the system can be efficiently solved using the Conjugate Gradient method,
despite the non-symmetry of $\MBB$ evident from~\eqref{eq:linear_system}. This
observation is studied in more detail in the next Section.

\section{Solution by the Conjugate Gradient method}\label{sec:convergence}

We start our analysis with recasting the system~\eqref{eq:linear_system} into a
more convenient form, by employing a certain operator and the associated
sub-space introduced later. Note that for simplicity, the reference conductivity
is taken as $\MB{L^0} = \lambda \MBI$.

\begin{definition}\label{def:operator} 
Given $\lambda > 0$, we define operator
$
\MB{P}_{\cE} 
=
\lambda\MB{F}^{-1}\MBhG\MB{F}
$
and associated sub-space as
$$
\cE 
= 
\left\{
\MB{P}_{\cE} 
\MBx
\mathrm{~for~}
\MBx \in\set{R}^{d\times \VN}
 \right\}
\subset
\set{R}^{d\times \VN}.
$$
\end{definition}

\begin{lemma}\label{lem:PE_projection}
The operator $\MB{P}_{\cE}$ is an orthogonal projection.
\end{lemma}
\bpro 
First, we will prove that $\MB{P}_{\cE}$ is projection, i.e.
$\MB{P}_{\cE}^2=\MB{P}_{\cE}$. Since $\MB{F}$ is a unitary matrix,
it is easy to see that
\begin{equation}\label{eq:projection_PE}
\MB{P}_{\cE}^2
=
(\lambda\MB{F}^{-1}\MB{\hG}\MB{F})(\lambda\MB{F}^{-1}\MB{\hG}\MB{F})
= \MB{F}^{-1}(\lambda\MBhG)^2\MB{F}.
\end{equation}
Hence, in view of the block-diagonal character of $\MB{\hG}$, it it sufficient
to prove the projection property of sub-matrices $(\lambda\MBhG)^{\Vk\Vk}$ only.
This follows using a simple algebra, recall Eq.~\eqref{eq:Gamma}:
$$
(\lambda\MBhG)^{\Vk\Vk}
(\lambda\MBhG)^{\Vk\Vk}
=
\frac{\Vxi(\Vk)\otimes\Vxi(\Vk)}{\Vxi(\Vk)\cdot\Vxi(\Vk)}
\cdot
\frac{\Vxi(\Vk)\otimes\Vxi(\Vk)}{\Vxi(\Vk)\cdot\Vxi(\Vk)}
=
\frac{\Vxi(\Vk)\otimes\Vxi(\Vk)}{\Vxi(\Vk)\cdot\Vxi(\Vk)}
=
(\lambda\MBhG)^{\Vk\Vk}.
$$
The orthogonality of $\MB{P}_{\cE}$ now follows \JV{from}
$$
\MB{P}_{\cE}^*
=
\left(
  \lambda\MB{F}^{-1}\MBhG\MB{F}
\right)^*
=
\lambda\MB{F}^{*}
\MBhG^*
\left(\MB{F}^{-1}\right)^*
=
\lambda\MB{F}^{-1}\MBhG\MB{F}
=
\MB{P}_{\cE},
$$
according to a well-known result of linear algebra, e.g. Proposition 1.8
in~\cite{Saad:2003:IMSL}.
\qed
\epro

\begin{remark}
It follows from the previous results that the subspace $\cE$ collects the
non-zero coefficients of trigonometric polynomials $\mathcal{T}^\VN$ with zero
rotation, which represent admissible solutions to the unit cell problem defined
by~\eqref{eq:problem_def_1}. Note that the orthogonal space $\cEp$ contains the
trigonometric representation of constant fields, cf.~\cite[Section~12.7]{Milton:2002:TC}.
\end{remark}

\begin{lemma}\label{lem:solution}
The solution $\MBe$ to the linear system~\eqref{eq:linear_system} admits the
decomposition $\MBe=\MB{e^0}+\MBe_{\cE}$, with $\MBe_{\cE}\in\cE$ satisfying 
\begin{equation}\label{eq:linear_system_cE2}
\MBPE\MBL\MBe_{\cE} 
+ 
\MBPE\MBL\MB{e^0} 
= 
\MB{0}.
\end{equation}
\end{lemma}

\begin{proof}
As $\MBe \in \set{R}^{d \times \VN}$, Lemma~\ref{lem:PE_projection} ensures that
it can be decomposed into two orthogonal parts $\MBe_{\cE} = \MBPE\MBe$ and
$\MBe_{\cEp} = (\MBI-\MBPE)\MBe$. Substituting this expression
into~\eqref{eq:linear_system}, and using the identity $\MB{B} =
\lambda\MB{F}^{-1}\MBhG\MB{F}\left(\tfrac{\MB{L}}{\lambda}-\MB{I}\right)$, we
arrive at
\begin{equation}\label{eq:linear_system_cE}
\frac{1}{\lambda}\MBPE\MBL\MBe_{\cE} 
+ 
\MBe_{\cEp} 
+ 
\frac{1}{\lambda}\MBPE\MBL\MBe_{\cEp} 
= 
\MB{e^0}.
\end{equation}
Since $\MBe^0 \in \cEp$, we have $\MBe_{\cEp} = \MB{e^0}$ and the proof is
complete. \qed
\end{proof}

With these auxiliary results in hand, we are in the position to present our
main result.

\begin{proposition}
  The non-symmetric system of linear equations~\eqref{eq:linear_system} is
  solvable by the Conjugate Gradient method for an initial vector $\MBe_{(0)}
  = \MBe^0 + \widetilde{\MBe}$ with $\widetilde{\MBe}\in\cE$. Moreover, the
  sequence of iterates is independent of the parameter $\lambda$.
\end{proposition}

\begin{proof}[outline]
It follows from Lemma~\ref{lem:solution} that the solution
to~\eqref{eq:linear_system} admits yet another, optimization-based,
characterization in the form
\begin{equation}
\MBe 
= 
\MBe^0 
+
\arg \min_{%
  \bar{\MBe}\in\cE
}
\left[
\frac{1}{2}
\scal{%
  \MBL\bar{\MBe}}{\bar{\MBe}
}_{\set{R}^{d\times\VN}}
+
\scal{%
\MBL\MBe^0}{\bar{\MBe}}_{\set{R}^{d\times\VN}}
\right].
\end{equation}
The residual corresponding to the initial vector $\MBe_{(0)}$ \JV{equals to}
$$
\MB{r}_{(0)} 
= 
\MBe^0 
- 
\left( \MBI + \MBB \right)
\left( \MBe^0 + \widetilde{\MBe} \right) 
=
-\frac{1}{\lambda} \MBPE\MBL\MBe^0
-\frac{1}{\lambda} \MBPE\MBL\widetilde{\MBe}
\in
\cE.
$$
It can be verified that the subspace $\cE$ is $\MBB$-invariant, thus
$(\MBI+\MBB)\cE\subset\cE$. Therefore, the Krylov subspace
$$
\Kryl{m}{\MBI+\MBB}{\MB{r}_{(0)}}
= 
\mathrm{span}  
\left\{
\Vr_{(0)}, 
(\MBI+\MBB)\MB{r}_{(0)},
\ldots,
(\MBI+\MBB)^m\MB{r}_{(0)}
\right\}
\subset \cE
$$
for arbitrary $m \in \mathbb{N}$. This implies that the residual $\MB{r}_{(m)}$
and the Conjugate Gradient search direction $\MB{p}_{(m)}$ at the $m$-th
iteration satisfy $\MB{r}_{(m)} \in \cE$ and $\MB{p}_{(m)} \in \cE$. Since
$\MBB$ is symmetric and positive-definite on $\cE$, the convergence of $\CG$
algorithm now follows from standard arguments, e.g.~Theorem~6.6 in \cite{Saad:2003:IMSL}. 
Observe that different choices of $\lambda$ generate identical Krylov subspaces,
thus the sequence of iterates is independent of $\lambda$.\qed
\end{proof}

\begin{remark}
Note that it is possible to show, using direct calculations based on the
projection properties of $\MBPE$, that the Biconjugate Gradient algorithm
produces exactly the same sequence of vectors as the Conjugate Gradient method,
see ~\cite{Vondrejc:2009:AHM}.
\end{remark}

\section{Numerical example}\label{sec:example}
To support our theoretical results, we consider a three-dimensional model
problem of electric conduction in a cubic periodic unit cell
$\PUC=\prod_{\alp=1}^3(-\frac{1}{2},\frac{1}{2})$, representing a
two-phase medium with spherical inclusions of $25\%$ volume fraction.
The conductivity parameters are defined as
\begin{equation*}
\mL(\Vx) = 
\begin{cases}
\rho\mI,&\|\Vx\|_2<(\frac{3}{16\pi})^{\frac{1}{3}}\\
\begin{pmatrix}
  1 & 0.2 & 0.2 \\
0.2 & 1 & 0.2 \\
0.2 & 0.2 & 1 \\
\end{pmatrix},&\text{otherwise}
\end{cases}
\end{equation*}
where $\rho>0$ denotes the contrast of phase conductivities. 
We consider the macroscopic field $\Ve^0 = [1,0,0]$ and discretize the unit cell
with $\VN = [n, n, n]$ nodes\footnote{%
In particular, $n$ was taken consequently as $16, 32, 64, 128$ and $160$ leading
up to $3\cdot 160^3 \doteq 12.2 \times 10^6$ unknowns}. The conductivity of the
homogeneous reference medium \JV{$\mL^0\in\set{R}^{d\times d}$} is parametrized as
\begin{align}\label{eq:ref_media_choice}
\mL^0 &= \lambda\mI, &
\lambda
  &= 
  1 - \omega + \rho \omega,
\end{align}
where $\omega \approx 0.5$ delivers the optimal convergence of the original
Moulinec-Suquet \JV{Fast-Fourier Transform-based Homogenization}~($\FFTH$)
algorithm~\cite{Moulinec:1994:FNMC}.

We first investigate the sensitivity of Conjugate Gradient~($\CG$) algorithm to
the choice of reference medium. The results appear in \Fref{fig:example1a},
plotting the relative number of iterations for $\CG$ against the conductivity of
the reference medium parametrized by $\omega$,
recall~\Eqref{eq:ref_media_choice}. As expected, $\CG$ solver achieve a
significant improvement over $\FFTH$ method as it requires about \JV{$40\%$
iterations of $\FFTH$ for a mildly-contrasted composite down to $4\%$ for
$\varrho = 10^3$}. The minor differences visible especially for $\rho = 10^3$
can be therefore attributed to accumulation of round-off errors. These
observations fully confirm our theoretical results presented earlier in
Section~\ref{sec:convergence}.

\begin{figure}[ht]
\centering
\subfigure[]{\includegraphics[scale=0.95]{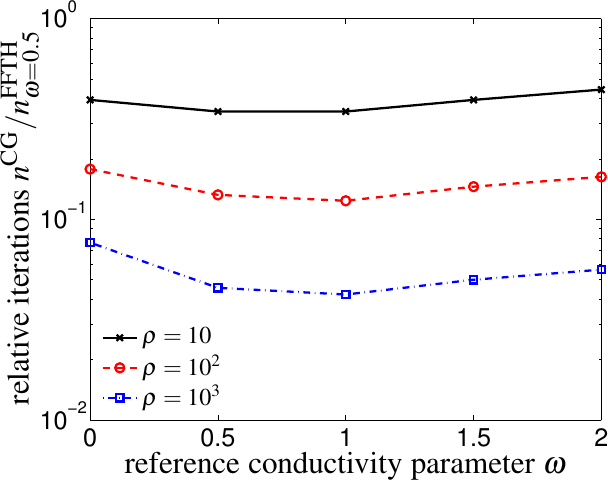}
\label{fig:example1a}
}
~
\subfigure[]{\includegraphics[scale=0.95]{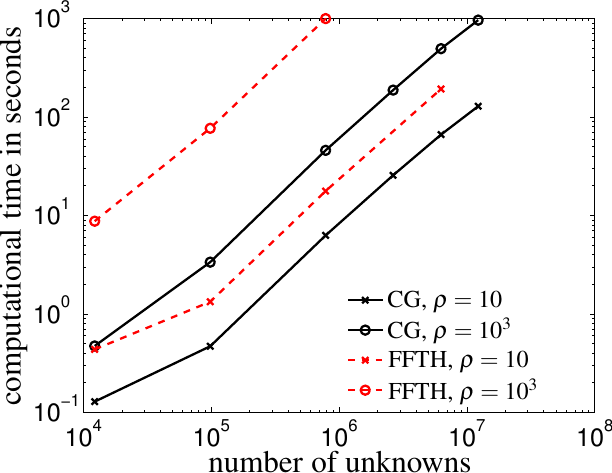}%
\label{fig:example1b}}
\caption{(a)~Relative number of iterations as a function of the
  reference medium parameter $\omega$ and (b)~computational time as a function
  of the number of unknowns.}
\label{fig:example1}
\end{figure}

In \Fref{fig:example1b}, we present the total computational
time\footnote{%
The problem was solved with a \Matlab in-house code on a machine
Intel\textsuperscript{\textregistered} Core\texttrademark 2 Duo 3 GHz CPU, 3.28
GB computing memory with Debian linux 5.0 operating system.}
as a function of the number of degrees of freedom and the phase ratio $\rho$.
The results confirm that the computational times scales linearly with the
increasing number of degrees of freedom for both schemes for a fixed
$\rho$~\cite{Zeman:2010:AFFT}. The ratio of the computational time for CG and
FFTH algorithms remains almost constant, which indicates that the cost of a
single iteration of $\CG$ and $\FFTH$ method is comparable. 

\JV{In addition, the memory requirements of both schemes are also comparable.
This aspect represents the major advantage of the short-recurrence CG-based
scheme over alternative schemes for non-symmetric systems, such as GMRES.
Finally note that finer discretizations can be treated by a straightforward
parallel implementation.}

\section{Conclusions}\label{sec:conclusions}
In this work, we have proven the convergence of Conjugate Gradient method for a
non-symmetric system of linear equations arising from periodic unit cell
homogenization problem and confirmed it by numerical experiment. The important
conclusions to be pointed out are as follows:

\begin{enumerate}

\item The success of the Conjugate Gradient method follows from the projection
properties of operator $\MB{P}_{\cE}$ introduced in
Definition~\ref{def:operator}, which reflect the structure of the underlying
physical problem.

\item Contrary to all available extensions of the $\FFTH$ scheme, the performance of
the Conjugate Gradient-based method is independent of the choice of reference
medium. This offers an important starting point for further improvements of
the method.

\end{enumerate}
\JV{Apart from the already mentioned parallelization, performance of the scheme
can further be improved by a suitable preconditioning procedure. This topic is
currently under investigation.}

\paragraph{Acknowledgments}
This work was supported by the Czech Science Foundation, through projects
No.~GA\v{C}R 103/09/1748, No.~GA\v{C}R 103/09/P490, No.~GA\v{C}R 201/09/1544,
and by the Grant Agency of the Czech Technical University in Prague through
project No.~SGS10/124/OHK1/2T/11.

\end{document}